\documentclass[12pt]{amsart}
\usepackage{amsmath,amssymb,latexsym,soul,cite,amsthm,color,enumitem,graphicx,tikz, mathtools,microtype}
\usepackage[colorlinks=true,urlcolor=airforceblue,citecolor=airforceblue,linkcolor=airforceblue,linktocpage,pdfpagelabels,bookmarksnumbered,bookmarksopen]{hyperref}
\definecolor{airforceblue}{rgb}{0.36, 0.54, 0.66}
\usepackage[english]{babel}
\usepackage[left=3cm,right=3cm,top=2.8cm,bottom=2.8cm]{geometry}

\numberwithin{equation}{section}

\newtheorem{theorem}{Theorem}[section]
\theoremstyle{plain}
\newtheorem{lemma}[theorem]{Lemma}
\theoremstyle{plain}
\newtheorem{proposition}[theorem]{Proposition}
\theoremstyle{plain}
\newtheorem{corollary}[theorem]{Corollary}

\theoremstyle{definition}
\newtheorem{remark}[theorem]{Remark}

\newcommand{\N}{{\mathbb N}}

\newcommand{\R}{{\mathbb R}}
\renewcommand{\H}{{\mathbb H}}
\renewcommand{\S}{{\mathbb S}}
\newcommand{\eps}{\varepsilon}
\newcommand{\beq}{\begin{equation}}
\newcommand{\eeq}{\end{equation}}
\renewcommand{\le}{\leqslant}
\renewcommand{\ge}{\geqslant}

\newcommand{\cal}{\mathcal}

\def\Xint#1{\mathchoice
{\XXint\displaystyle\textstyle{#1}}%
{\XXint\textstyle\scriptstyle{#1}}%
{\XXint\scriptstyle\scriptscriptstyle{#1}}%
{\XXint\scriptscriptstyle\scriptscriptstyle{#1}}%
\!\int}
\def\XXint#1#2#3{{\setbox0=\hbox{$#1{#2#3}{\int}$ }
\vcenter{\hbox{$#2#3$ }}\kern-.6\wd0}}

\def\dashint{\Xint-}

\title[Liouville theorems for ancient caloric funtions]{Liouville theorems for ancient caloric functions via optimal growth conditions}

\author[S.\ Mosconi]{Sunra Mosconi}

\address[S.\ Mosconi]{Dipartimento di Matematica e Informatica
\newline\indent
Universit\`a degli Studi di Catania
\newline\indent
Viale A.\ Doria 6, 95125 Catania, Italy}
\email{mosconi@dmi.unict.it}

\subjclass[2010]{}

\begin{document}

\begin{abstract}
We provide some Liouville theorems for ancient nonnegative solutions of the heat equation on a complete non-compact Riemannian manifold with Ricci curvature bounded from below. We determine growth conditions ensuring triviality of the latters, showing their optimality through examples. 
\end{abstract}

\maketitle

\section{Introduction}\label{sec1}
Two instances of Liouville's type theorems are the following well known result:
\begin{enumerate}
\item
If $u$ is harmonic on $\R^{N}$ and is bounded from below, then it is constant.
\item
If $u$ is harmonic on $\R^{N}$ and grows sublinearly at infinity, then it is constant.  
\end{enumerate}
The first statement follows from the Harnack inequality, while the second one from gradient estimates for harmonic functions. These kind of results received ever increasing attention in the last decades, with generalizations to other  partial differential equations, Riemannian manifolds under Ricci curvature lower bounds (see the recent survey \cite{CM}), or even to metric measure spaces . In this paper, motivated by \cite{SZ, LZ}, we investigate various form of Liouville theorems for the {\em heat equation} on a {\em Riemannian manifold}, with the aim to find optimal growth/bound conditions ensuring triviality of its solutions. In the following, we refer to these kind of conditions as {\em Liouville properties}: thus, a one-sided bound is a Liouville property for the Laplacian on $\R^{N}$, as well as sublinearity.

In the parabolic setting, Liouville properties are more subtle. On the one hand, solving the Cauchy problem in $\R^{N}$ with initial datum in $C^{\infty}_{c}(\R^{N})$ shows that no global growth/bound condition can ensure triviality of a solution $u\in C^{\infty}(\R^{N}\times \, ]0, +\infty[)$. One can seek for a  (sub-exponential) {\em decay} at infinity as a Liouville property, but we won't pursue this investigation here.  Instead, we identify the lack of a strong Liouville property with the lower boundendess in time of the domain. Indeed, an immediate byproduct of the parabolic Harnack inequality is the constancy of any bounded solution of the heat equation on $\R^{N}\times \ ]-\infty, 0[$. Solutions of the heat equation defined in $]-\infty, T[$ are called {\em ancient} if $T<+\infty$, {\em eternal} if $T=+\infty$, while nonnegative (or, for all that matter, bounded from below) ones are called {\em caloric}. To explore further the classical case of $\R^{N}$ we consider the two main examples of ancient (actually, eternal) solutions, namely
\begin{equation}
\label{examples}
u(x, t):=e^{x_{N}+t}, \qquad v(x, t)=e^{-t}\cos x_{N},\qquad x=(x_{1}, \dots, x_{N})\in \R^{N}
\end{equation}
The first example shows that, even in the setting of eternal solutions, non-negativity is {\em not} a Liouville property. The second shows that  boundedness {\em at fixed time} also fails to be a Liouville property for the heat equation on $\R^{N}$. The best  parabolic Liouville  theorem in $\R^{N}$ dates back to Hirschman \cite{H52} (see also \cite{W}). We state it and give a modern proof taken from \cite{DMV}. 
\begin{theorem}[Hirshman]\label{Hth}
Let $u$ be an ancient caloric function on $\R^{N}\times \ ]-\infty, T_{0}[$ such that  for a fixed time $t_{0}<T_{0}$ it holds
\[
u(x, t_{0})\leq e^{o(|x|)},\qquad \text{for $|x|\to +\infty$}
\]
Then, $u$ is constant.
\end{theorem}
 
\begin{proof} We can assume that $u>0$. By the Widder representation for ancient positive solutions (see \cite{LZ}), there exists a non-negative Borel measure $\mu$ such that  
\begin{equation}
\label{W}
u(x, t)=\int_{\R^{N}} e^{x\cdot \xi+t|\xi|^{2}}\, d\mu(\xi).
\end{equation}
By H\"older inequality with respect to the measure $\nu:=e^{t_{0}|\xi|^{2}}\mu$ 
\[
\begin{split}
u(sx+(1-s)y, t_{0})&=\int_{\R^{N}} e^{(sx+(1-s)y)\cdot \xi}\, d\nu(\xi)\le \left(\int_{\R^{N}} e^{x\cdot \xi}\, d\nu(\xi)\right)^{s}\left(\int_{\R^{N}} e^{y\cdot \xi}\, d\nu(\xi)\right)^{1-s}\\
&=u^{s}(x, t_{0})\, u^{1-s}(y, t_{0}),
\end{split}
\]
 for all $s\in\, ]0, 1[$. Therefore,  $x\mapsto \log u(x, t_{0})$ is convex, and being sublinear by assumption, it must be constant. Thus $u(x, t_{0})\equiv c$ and differentiating under the integral sign \eqref{W}, we obtain
\[
0=\left.P(D_{x})u(x, t_{0})\right|_{x=0}=\int_{\R^{N}} P(\xi)\, d\nu(\xi)
\]
for any polinomial $P$ such that $P(0)=0$. By the Stone-Weierestrass and Riesz representation theorems, this implies that ${\rm supp}(\nu)=\{0\}$ and thus $\mu=c\,\delta_{0}$ for some $c\in \R$. Inserting the latter into \eqref{W} gives the claim. 
\end{proof}

The two examples in \eqref{examples} show that the assumption in the previous Liouville theorem are optimal.

The picture becomes more involved if we substitute $\R^{N}$ with a general Riemannian manifold. Indeed, if $\H_{N}$ denotes the real hyperbolic space of dimension $N$, there are plenty of bounded harmonic functions\footnote{Indeed, classical harmonics  on $B_{1}$ are also harmonic on $\H_{N}$ identified with the Poincar\'e disc.}, which are also eternal solutions of $\partial_{t}-\Delta=0$. It turns out, however, that this issue can only appear in negative curvature and the following is the more general Liouville type theorem in the Riemannian framework up to now.

\begin{theorem}[Souplet-Zhang \cite{SZ}]\label{SZth}
Let $M$ be a complete  Riemannian manifold with non-negative Ricci curvature, $p\in M$ and ${\rm d}$ be the metric distance. Any ancient caloric $u$ such that $u(x, t)\le e^{o({\rm d}(x, p)+\sqrt{-t})}$ for ${\rm d}(x, p), -t\to +\infty$,  is constant.
\end{theorem}
This result was actually an immediate consequence of a new gradient estimate for positive  solutions $u$ of the heat equation in $Q_{R, T}:=B_{R}(p)\times [-T, 0]$, $p\in M$. If ${\rm Ric}_{M}\ge -K\, g$ for some $K\ge 0$, then the gradient estimate of \cite{SZ} states that 
\begin{equation}
\label{szin}
|\nabla \log u|^{2}\le C_{N}\left(K+\frac{1}{R^{2}}+\frac{1}{T}\right)\log^{2} (S/u),\qquad S=\sup_{Q_{R, T}}u
\end{equation}
holds in $Q_{R/2, T/2}$ for any $u$ as above. The previous inequality falls into the wider framework of {\em parabolic gradient estimates} such as the celebrated Li-Yau's one \cite{LY}
\begin{equation}
\label{lyin}
|\nabla \log u|^{2}\le \partial_{t}\log u+ C_{N}\left(K+\frac{1}{R^{2}}+\frac{1}{T}\right)
\end{equation}
or the Hamilton inequality \cite{H93} (generalized by Kotschvar \cite{K} to the non-compact case)
\begin{equation}
\label{hkin}
|\nabla \log u|^{2}\le C_{N}\left(K+\frac{1}{T}\right)\log(S/u), \qquad S=\sup_{M\times \, [-T, 0]}u.
\end{equation}
Notice that letting $T\to +\infty$ into this last inequality immediately gives  a Liouville theorem for {\em bounded} ancient caloric functions in the case $K=0$,
but it is only its localized counterpart \eqref{szin} which provides the much weaker sub-exponential growth condition $u\le e^{o({\rm d}(x, p)+\sqrt{-t})}$ as a Liouville property.

Theorem \ref{SZth} has been generalized (with the same growth condition) in various directions, see e.g. \cite{Huang} and the bibliography therein. The main result of this note is the following improvement of Theorem \ref{SZth}.

\begin{theorem}\label{mainth}
Let $M$ be a complete  Riemannian manifold with Ricci curvature bounded from below by $-K\le 0$,  $p\in M$ and $u$ be an ancient caloric function. 
\begin{enumerate}
\item
If $K=0$ and $u(x, t_{0})\le e^{o({\rm d}(x, p))}$ for ${\rm d}(x, p)\to +\infty$ at some fixed $t_{0}$,  $u$ is constant.
\item
If $K>0$ and $u(x, t)\le e^{o({\rm d}(x, p)-t)}$ for ${\rm d}(x, p) -t\to +\infty$,  $u$ is stationary (and hence harmonic).
\end{enumerate}
\end{theorem}
Let us make some comments on the result. The case $K=0$ fixes the gap between the Euclidean Liouville Theorem \ref{Hth} and the Riemannian Liouville Theorem \ref{SZth}, providing an {\em optimal} parabolic Liouville property in the case ${\rm Ric}_{M}\ge 0$ more in the spirit of \cite{KPP}. It is worth noting that  an argument employing \eqref{lyin} in conjunction with \eqref{szin} would give the same result; however, we will proceed in a softer way via a Choquet representation for ancient solutions (see Lemma \ref{lemma1} below). This method gives us the opportunity to treat the case of negative Ricci curvature lower bounds as well. The case  $M=\H_{N}$ discussed above shows that our second Liouville statement is the best one can get through a growth condition. We will actually prove a slightly more general statement in the case of negative Ricci curvaure bound and refer to Remark \ref{finalR} for a description. 
\smallskip

As a final remark, Liouville properties for ancient caloric functions in the metric measure setting of ${\rm RCD}^{*}(K, N)$ spaces can probably be obtained through the same techniques described here, providing a generalization of the $RCD^{*}(K, N)$ counterpart of  Theorem \ref{SZth}. The latter has been proved in the metric measure setting in \cite{Huang} through a gradient estimate of the form \eqref{szin}, but the more general statement of Theorem \ref{mainth} in the ${\rm RCD}^{*}$ framework can be cooked up via the same ingredients: the granted linearity of the Laplacian is essential in order to apply Choquet theory, the parabolic Harnack inequality holds true since ${\rm RCD}^{*}(K, N)$ verifies doubling and Poincar\'e, while the relevant Laplacian comparison and comparison principles can be found in \cite{GM}. The only additionally needed result is a gradient estimate for eigenfunctions of Yau's type (see Proposition \ref{borb} below), which follows from the metric version of the parabolic Li-Yau inequality proved in \cite{ZZ}.

\section{Proof of the main result}

By a time translation will always work with caloric functions on $M\times \, ]-\infty, 1[$. If $u$ is ancient and caloric (i.\,e., a non-negative ancient solution of the heat equation) then the local Harnack inequality shows that if $u(x_{0}, t_{0})=0$, then $u$ vanishes identically on $M\times \, ]-\infty, t_{0}]$. Since we are supposing that ${\rm Ric}_{M}$ is bounded from below, uniqueness of the non-negative Cauchy problem holds (see e.\,g. \cite{M} and the references therein), and therefore any nontrivial caloric function is strictly positive.  Let ${\cal C}$ be the cone of caloric functions on $M\times \, ]-\infty, 1[$. We say that $u\in {\cal C}$ is {\em minimal}, and write $u\in {\rm Ext}({\cal C})$, if 
\[
v\in {\cal C}\quad \text{and}\quad v\le u\quad \Rightarrow\quad v=k\, u \quad \text{for some $k\in \R$}.
\]
The following result is basically contained in \cite{KT, P}, see also \cite[Remark 2.3]{LZ}. 
\begin{proposition}[Extremal caloric functions]\label{extremal}
Let $M$ be a complete Riemannian manifold with Ricci curvature bounded from below. If $v\in {\rm Ext}({\cal C})$ then 
there exists $\lambda\in \R$ and $w\in C^{\infty}(M)$ solving $\Delta w=\lambda\, w$ such that $v(x, t)=e^{\lambda\, t}\, w(x)$.
\end{proposition}
Regarding eigenfunctions, we recall the following a-priori bound of \cite{B}, which is a refinement of the classical Yau's gradient estimate in \cite{Yau}.

\begin{proposition}[Gradient bound for eigenfunctions]\label{borb}
Let $(M, g)$ be a complete $N$-dimensional Riemannian manifold with ${\rm Ric}_{M}\ge -(N-1)\, \kappa\, g$ for $\kappa\ge 0$ and $w$ a positive $\lambda$-eigenfunction. Then $\lambda\ge -(N-1)^{2}\, \kappa/4$ and 
\[
|\nabla \log w|\le \frac{N-1}{2}\left( \sqrt{\kappa +\frac{4\,\lambda}{(N-1)^{2}}}+\sqrt{\kappa}\right)
\] 

\end{proposition}

In particular, by Yau's elliptic Liouville theorem, if ${\rm Ric}_{M}\ge 0$, positive nontrivial solutions of $\Delta w=\lambda\, w$ exist only for $\lambda>0$. 

\begin{lemma}\label{lemma1}
Let $M$ be a complete Riemannian manifold with Ricci curvature bounded from below and $u\in {\cal C}$. There exists a probability measure $\nu$ on $\R$ such that 
\begin{equation}
\label{repre}
u(x, t)=\int_{\R} e^{\lambda\, t}\, w_{\lambda}(x)\, d\nu
\end{equation}
where, for $\nu$ a.\,e. $\lambda$, $w_{\lambda}$ is a positive solution of $\Delta w=\lambda\, w$.
\end{lemma}

\begin{proof}
The cone ${\cal C}$ fails to have a compact base with respect to any useful topology. However, if we equip it with the topology of pointwise convergence, it turns out to be a proper closed a subset of $\R^{M\times\, ]-\infty, 1[}$ and therefore is weakly complete. We claim that ${\cal C}$ is metrizable and hence well-capped in the Choquet sense (see \cite[30.16]{C}). Indeed, let $D\subseteq M\times\, ]-\infty, 1[$ be denumerable  and dense. The local parabolic Harnack inequality implies that the topology of pointwise convergence in $D$ coincides with the pointwise convergence in ${\cal C}\subseteq \R^{M\times \, ]-\infty, 1[}$ (it actually  implies locally uniform convergence). This proves metrizability due to $D$ being denumberable and, even more, that ${\cal C}$ is second countable and thus separable. As a consequence, ${\cal C}$ is a Polish space, and being ${\rm Ext}({\cal C})$ a $G_{\delta}$ subset of ${\cal C}$\footnote{With little effort, it is possible to prove that it is closed.}, it turns out to be Polish as-well.

By Choquet theorem \cite[Theorem 30.22]{C}, any $u\in {\cal C}$ can be represented through a probability measure supported on ${\rm Ext}({\cal C})$, i.e. there exists a probability measure $\mu$ on ${\rm Ext}({\cal C})$ such that for any continuous linear functional $\Lambda$ 
\[
\langle \Lambda, u\rangle=\int_{{\rm Ext}({\cal C})}\langle\Lambda, v\rangle\, d\mu.
\]
Specifying  $\Lambda$ to be the evaluation at $(x, t)\in M\times\, ]-\infty, 1[$, gives
\[
u(x, t)= \int_{{\rm Ext}({\cal C})}v(x, t)\, d\mu\qquad \forall (x, t)\in M\times\, ]-\infty, 1[. 
\]
Let us fix $p\in M$ and observe that the map $\psi: {\rm Ext}({\cal C})\to \R$ between Polish spaces defined as 
\[
\psi(v)=v^{-1}(p, 0)\, \frac{\partial v}{\partial t} (p, 0)
\]
is measurable\footnote{It is actually continuous by local parabolic regularity, but we wont need it.} and thus induces a disintegration of the probability measure $\mu$ into probability measures  $\{\mu_{\lambda}\}_{\lambda}$, Borel measurable with respect to $\lambda$, such that ${\rm supp}(\mu_{\lambda}) \subseteq \psi^{-1}(\lambda)$. In particular, the Disintegration theorem ensures that there exists a probability measure $\nu$ on $\R$ such that
\begin{equation}
\label{repr}
u(x, t)=\int_{\R}\int_{\psi^{-1}(\lambda)} v(x, t)\, d\mu_{\lambda}\, d\nu.
\end{equation}
By Proposition \ref{extremal}, any $v\in {\rm Ext}({\cal C})$ is of the form $v(x, t)=e^{\lambda\, t}\, w(x)$ for some $\lambda\in \R$ and $w\ge 0$ solving $\Delta w=\lambda\, w$, therefore it holds
\begin{equation}
\label{p}
v(x, t)=e^{\lambda\, t}\, v(x, 0),\qquad \psi(v)=\lambda.
\end{equation} 
If ${\cal C}_{\lambda}$ denotes the cone of non-negative solutions to $\Delta w=\lambda\, w$, the latter discussion shows that
\[
\psi^{-1}(\lambda)\subseteq \{v: v(x, t)=e^{\lambda\, t}\, w(x), w\in {\cal C}_{\lambda}\}.
\]
The map $\Phi_{\lambda}:\psi^{-1}(\lambda)\to {\cal C}_{\lambda}$ defined as  $\Phi_{\lambda}(v)(x)=v(x, 0)$
is continuous and induces a push-forward measure $(\Phi_{\lambda})_{*}(\mu_{\lambda})$ on ${\cal C}_{\lambda}$, which we still denote by $\mu_{\lambda}$ by a slight abuse of notation. By construction, it satisfies
\[
\int_{\psi^{-1}(\lambda)} v(x, 0)\, d\mu_{\lambda}=\int_{{\cal C}_{\lambda}}w(x)\, d\mu_{\lambda}=:w_{\lambda}(x)
\]
and we observe that, being $\lambda\mapsto \mu_{\lambda}$ Borel, so is $\lambda\mapsto w_{\lambda}$ and thus $(\lambda, x)\mapsto w_{\lambda}(x)$. 
Using the distributional formulation of the equation $\Delta w=\lambda\, w$, namely
\[
\int_{M}w\, \Delta\varphi\, dg=\lambda\int_{M}w\, \varphi\, dg,\qquad \forall \varphi\in C^{\infty}_{c}(M)
\]
 and Fubini-Tonelli's theorem, it is readily checked that $w_{\lambda}$ is a distributional, and thus classical solution of $\Delta w=\lambda\, w$. Finally, by the first relation in \eqref{p}, 
\[
\int_{\psi^{-1}(\lambda)} v(x, t)\, d\mu_{\lambda}=e^{\lambda\, t}\, w_{\lambda}(x)
\]
and recalling \eqref{repr} completes the proof of \eqref{repre}. Finally, the strong minimum principle ensures that  $w_{\lambda}(x)=0$ for some $x\in M$ implies $w_{\lambda}\equiv 0$, so that it suffices to restrict $\nu$ to the measurable subset $\{w_{\lambda}(p)>0\}$ for a fixed $p$. 
\end{proof}

\begin{lemma}\label{lemma2}
Let $(M, g)$ be a complete  $N$-dimensional Riemannian manifold with ${\rm Ric}_{M}\ge -(N-1)\, \kappa\, g$, $
\kappa\ge 0$ and  $p\in M$. For $\lambda>0$ define 
\begin{equation}
\label{chilambda}
\chi_{\lambda}=\chi_{\lambda}(\kappa, N)=\frac{N-1}{2}\left(\sqrt{\kappa+\frac{4\, \lambda}{(N-1)^{2}}}-\sqrt{\kappa}\right).
\end{equation}
There exists $\bar w_{\lambda}\in {\rm Lip}_{\rm loc}(M)$ such that $\bar w_{\lambda}(p)=1$, $\Delta \bar w_{\lambda}\le \lambda\, \bar w_{\lambda}$ weakly on $M$ and 
\begin{equation}
\label{barr}
\begin{cases}
\bar w_{\lambda}(x)\ge c(N)\, e^{ \chi_{\lambda}\, d(x, p)}&\text{if $\kappa>0$}\\[5pt]
\bar w_{\lambda}(x)\ge c(N, \eps)\, e^{ (\chi_{\lambda}-\eps)\, d(x, p)}&\text{for any $\eps>0$, if $\kappa=0$},
\end{cases}
\end{equation}
for $x\in M$, where $c(N)$ and $c(N, \eps)$ are suitable positive constants.
\end{lemma}

\begin{proof}
We begin considering the case ${\rm Ric}_{M}\ge -(N-1)\, \kappa\, g$ for $\kappa>0$. Eventually rescaling the metric, we can suppose without loss of generality that $\kappa=1$. Moreover, the condition $\bar w_{\lambda}(p)=1$ can be dropped, as it suffices to eventually multiply by a suitable constant. Let $\H$ denote the real hyperbolic space of dimension $N$, with corresponding Laplace-Beltrami operator $\Delta_{\H}$ and distance ${\rm d}_{\H}$. We identify $\H$ with the open ball $B_{1}\subseteq \R^{N}$ equipped with the Poincar\'e metric $g=4\, (1-|z|^{2})^{-2}{\rm Id}$, obtaining in particular
\begin{equation}
\label{dH}
{\rm d}_{\H}(z, 0)=\log\left(\frac{1+|z|}{1-|z|}\right), \qquad \forall z\in B_{1}\subseteq \R^{N}.
\end{equation}
 The Busemann function $b_{\nu}$ for the geodesic ray $\gamma_{\nu}$ from $0$ with direction $\nu$, $|\nu|=1$ is explicitly given by
\[
b_{\nu}(z)=\lim_{t\to +\infty}{\rm d}_{\H}(\gamma_{\nu}(t), z)-t=-\log\left(\frac{1-|z|^{2}}{|z-\nu|^{2}}\right),
\]
and satisfies
\[
|\nabla_{\H} b_{\nu}|=1,\qquad \Delta_{\H}b_{\nu}=N-1.
\]
From the latters, we immediately compute 
\[
\Delta_{\H}\, e^{\mu\, b_{\nu}}=\mu\,(\mu+N-1)\,e^{\mu\, b_{\nu}}.
\]
For $\lambda>0$, we choose 
\[
\mu_{\lambda}:=\frac{1}{2}\big(1-N-\sqrt{(N-1)^{2}+4\, \lambda}\big)<1-N
\]
and let
\begin{equation}
\label{wlambda}
w_{\lambda, \nu}(z)=e^{\mu_{\lambda}\, b_{\nu}(z)}=\left(\frac{1-|z|^{2}}{|z-\nu|^{2}}\right)^{-\mu_{\lambda}},
\end{equation}
so that, being $\mu_{\lambda}\, (\mu_{\lambda}+N-1)=\lambda$, $w_{\lambda, \nu}$ is a positive $\lambda$-eigenfunction. 
Finally, we let 
\[
w_{\lambda}(z)=\int_{{\mathbb S}^{N-1}}w_{\lambda, \nu}(z)\, d{\cal H}^{N-1}(\nu),\qquad z\in B_{1}\subseteq \R^{N},
\]
which is again a positive $\lambda$-eigenfunction, radial by construction. As such, letting $f(r)=w_{\lambda}(z)$ with $r={\rm d}_{\H}(z, 0)$ and using polar hyperbolic coordinates, $f$ obeys
\[
f''(r)+(N-1)\,\tanh r\, f'(r)=\lambda\, f(r)\qquad f'(0)=0.
\]
Multiplying by $(\cosh r)^{N-1}$ both sides and integrating, we get
\[
(\cosh r)^{N-1}\, f'(r)=\int_{0}^{r}\left((\cosh \tau)^{N-1}\, f'(\tau)\right)'d\tau=\lambda\, \int_{0}^{r} (\cosh \tau)^{N-1}\, f(\tau)\, d\tau\ge 0
\]
which implies that $w_{\lambda}$ is radially increasing.
We claim that 
\begin{equation}
\label{claim}
 w_{\lambda}(z)\ge c_{N}\, e^{\chi_{\lambda}\, {\rm d}_{\H}(z, 0)}
\end{equation}
for $\chi_{\lambda}$ given in \eqref{chilambda}. Using $w_{\lambda, \nu}>0$ and the expression in \eqref{wlambda}, we get
\[
\begin{split}
\int_{\S^{N-1}}w_{\lambda, \nu}(z)\, d{\cal H}^{N-1}(\nu)&\ge \int_{\{\nu\in \S^{N-1}: |z-\nu|\le 2\,(1-|z|)\}}\left(\frac{1-|z|^{2}}{|z-\nu|^{2}}\right)^{-\mu_{\lambda}}\, d{\cal H}^{N-1}(\nu)\\
&\ge \left(\frac{1+|z|}{1-|z|}\right)^{-\mu_{\lambda}} {\cal H}^{N-1}\big(\big\{\nu\in \S^{N-1}: |z-\nu|\le 2\,(1-|z|)\big\}\big).
\end{split}
\]
Through an elementary geometric argument, we see that it holds
\[
{\cal H}^{N-1}\big(\big\{\nu\in \S^{N-1}: |z-\nu|\le 2\,(1-|z|)\big\}\big)\ge c_{N} \, (1-|z|)^{N-1}
\]
for some $c_{N}>0$, so that being $\mu_{\lambda}+N-1=-\chi_{\lambda}$,
\[
 w_{\lambda}(z)\ge  c_{N}\, \frac{(1+|z|)^{\chi_{\lambda}+N-1}}{(1-|z|)^{\chi_{\lambda}}}\ge c_{N}\left(\frac{1+|z|}{1-|z|}\right)^{\chi_{\lambda}}.
\]
Recalling formula \eqref{dH} for the distance ${\rm d}_{\H}$ proves \eqref{claim}. Finally, let $\bar w_{\lambda}\in {\rm Lip}_{loc}(M)$ be defined through
\[
\bar w_{\lambda}(x)= w_{\lambda}(z), \qquad \text{with} \quad {\rm d}_{\H}(z, 0)={\rm d}(x, p),
\] 
where ${\rm d}$ is the usual metric distance in $M$. Clearly $\bar w_{\lambda}$ is well defined by the radiality of $w_{\lambda}$ and \eqref{claim} holds true in $M$ as well by construction. Since ${\rm Ric}_{M}\ge - g$, and $w_{\lambda}$ is radially increasing, the Laplacian comparison implies that, weakly in $M$,
\[
\Delta \bar w_{\lambda}(x)\le \Delta_{\H} w_{\lambda}(z) =\lambda \, w_{\lambda}(z)=\lambda\, \bar w_{\lambda}(x).
\]
The case $\kappa=0$ is easier, as in the model space $\R^{N}$ we define
\[
w_{\lambda}(z)=\int_{\S^{N-1}}e^{\sqrt{\lambda}\, z\cdot \nu}\, d{\cal H}^{N-1}(\nu),\qquad z\in \R^{N},
\]
which is a radial positive $\lambda$-eigenfunction. It is radially increasing by integrating, as before, the corresponding ODE, so that it suffices to prove the pointwise lower bound. To this end, let $\eps>0$, $z=r\, e_{1}$, $e_{1}=(1,0, \dots, 0)$ and compute 
\[
\begin{split}
\int_{\S^{N-1}}e^{\sqrt{\lambda}\, z\cdot \nu}\, d{\cal H}^{N-1}(\nu)&\ge \int_{\{\nu \in \S^{N-1}:e_{1}\cdot \nu\ge (1-\eps)\}}e^{\sqrt{\lambda}\, z\cdot \nu}\, d{\cal H}^{N-1}(\nu)\\
&\ge e^{\sqrt{\lambda}\, (1-\eps)\, r}\, {\cal H}^{N-1}\big(\big\{\nu \in \S^{N-1}:e_{1}\cdot \nu\ge (1-\eps)\big\}\big)
\end{split}
\]
which proves the claimed lower bound when $M=\R^{N}$. Applying the Laplacian comparison as before, we get the claim.
\end{proof}

\begin{corollary}
Let  $M, p$  be as above. If $w>0$ solves $\Delta w=\lambda\, w$ for $\lambda>0$, then
\begin{equation}
\label{lb}
\liminf_{r\to +\infty}\sup_{\partial B_{r}(p)}\frac{\log w}{ r}\ge \chi_{\lambda}.
\end{equation}
\end{corollary}

\begin{proof}
The statemenet of the corollary in unaffected by multiplying $w$ for positive constants, so we can suppose that $w(p)=2$. If \eqref{lb} is false, there exists $\eps>0$ and a sequence $\{r_{n}\}$ with $r_{n}\to +\infty$ such that for any sufficiently large $n$ it holds $w\le {\rm exp}(\chi_{\lambda}\, (1-2\,\eps)\, r_{n})$ on $\partial B_{r_{n}}(p)$. Let $\bar w_{\lambda}$ and $c=c(N, \eps)$ be given the previous Lemma. By \eqref{barr}, for sufficiently large $n$ it holds
\[
\bar w_{\lambda}\ge c\, {\rm exp}(\chi_{\lambda}\, (1-\eps)\, r_{n})\ge {\rm exp}(\chi_{\lambda}\, (1-2\,\eps)\, r_{n})\ge w
\]
on $\partial B_{r_{n}}(p)$, so that the weak comparison principle for $-\Delta +\lambda$ in $B_{r_{n}}(p)$ implies $\bar w_{\lambda}(p)\ge w(p)$. As $\bar w_{\lambda}(p)=1$ and $w(p)=2$, this is a contradiction. 
\end{proof}

\medskip

\begin{proof} of Theorem \ref{mainth}.

Suppose that ${\rm Ric}_{M}\ge -(N-1)\, \kappa\, g$ with $\kappa\ge 0$, let $u$ be a caloric ancient solution in $M\times\, ]-\infty, 1[$ and consider the representation given in \eqref{repre} of  Lemma \ref{lemma1}. We will prove, separately for $\kappa=0$ and $\kappa>0$, that  the assumed growth conditions force in both cases ${\rm supp}(\nu)=\{0\}$. This in turn implies that $u$ is stationary and harmonic, concluding the proof in the case $\kappa>0$,  while  an application Yau's elliptic Liouville theorem will ensure $u\equiv c >0$ in the case $\kappa=0$. 
\medskip

{\em Case $\kappa=0$}.
Suppose by contradiction that ${\rm supp}(\nu)\ne \{0\}$. Recall that  ${\rm Ric}_{M}\ge 0$ implies that $M$ possesses  positive $\lambda$-eigenfunctions only for $\lambda\ge 0$, therefore we can suppose that there exists $0<a<b$ such that $\nu([a, b])>0$. \\
As pointed out in the proof of Lemma \ref{lemma1}, the function $[a, b]\ni \lambda\mapsto w_{\lambda}\in \R^{M}$ is Borel when $\R^{M}$ has the pointwise convergence topology. 
By Proposition \ref{borb}, the set 
\[
E:=\left\{\log w: w>0,\ \Delta w=\lambda\, w\ \text{for some $\lambda\in [a,b]$}\right\}
\]
is equilipschitz, so that pointwise convergence and locally uniform convergence coincide on $E$. We metrize the latter topology through
\begin{equation}
\label{metric}
{\rm d}_{E}(f, g)=\sup_{n}\left\{\frac{1}{n}\sup_{B_{n}(p)}|f-g|\right\}
\end{equation}
which is finite for any $f, g\in E$ due to the above mentioned equilipschitzianity. We apply Lusin theorem to 
\[
[a, b]\ni\lambda \mapsto \varphi(\lambda):=\log w_{\lambda}\in E
\]
(which is therefore Borel from $[a, b]$ to $(E, {\rm d}_{E})$), obtaining a compact $K\subseteq [a, b]$ such that 
\[
\nu(K)\ge \nu([a, b])/2,\qquad \varphi\lfloor_{K}\in C^{0}(K, (E, {\rm d}_{E})).
\]
Fix a point $\lambda_{0}\in K$ such that 
\[
\nu\big(I_{r}(\lambda_{0})\cap K\big)>0\qquad \forall r>0
\]
where $I_{r}(\lambda_{0})=[\lambda_{0}-r, \lambda_{0}+r]$ and let $\eps>0$ to be determined.  The continuity of $\varphi$ in $\lambda_{0}$ implies that there exists $r_{\eps}>0$ such that for any $n\ge 1$,
\[
\log w_{\lambda_{0}}\le \log w_{\lambda}+\eps\, n,\qquad \text{in $B_{n}(p)$, } \forall \lambda\in I_{r_{\eps}}(\lambda_{0})\cap K.
\]
Taking the mean value over $I_{r_{\eps}}(\lambda_{0})\cap K$ with respect to the measure $\nu$ and using Jensen inequality we infer that in the ball $B_{n}(p)$ it holds
\[
\log w_{\lambda_{0}}\le \dashint_{I_{r_{\eps}}(\lambda_{0})\cap K}\log w_{\lambda}\, d\nu+\eps\, n\le \log \left(\dashint_{I_{r_{\eps}}(\lambda_{0})\cap K}w_{\lambda}\, d\nu\right)+\eps\, n
\]
and by the positivity of $w_{\lambda}$ and the representation \eqref{repre} of $u$ we conclude
\[
\log w_{\lambda_{0}}\le -\log\nu\big(I_{r_{\eps}}(\lambda_{0})\cap K\big)+\log u+\eps\, n
\]
in $B_{n}(p)$. We take the supremum on $\partial B_{n}(p)$, divide by $n$ and let $n\to +\infty$. By the assumption $\log u(x)\le o({\rm d}(x, p))$, we get
\[
\limsup_{n}\sup_{\partial B_{n}}\frac{\log w_{\lambda_{0}}}{n}\le \eps,
\]
which gives a contradiction to \eqref{lb} if $\eps<\chi_{\lambda_{0}}$.
\medskip

{\em Case $\kappa>0$}. Consider as before the Choquet representation \eqref{repre} of $u$. The assumption $u(x, t)\le e^{o({\rm d}(x, p)-t)}$ as ${\rm d}(x, p)-t\to +\infty$ entails 
\begin{equation}
\label{fin}
u(p, t)=\int e^{\lambda\, t}\, w_{\lambda}(p)\, d\nu\le e^{o(-t)}.
\end{equation}
The latter in turn implies that $\nu(]-\infty, 0[)=0$, for otherwise, being $\lambda\mapsto w_{\lambda}(p)$ Borel, Lusin's theorem provides a compact  $K\subset \, ]-\infty, 0[$ such that 
\[
\nu(K)>0,\qquad \lambda\mapsto w_{\lambda}(p)\in C^{0}(K, \R).
\]
Since $w_{\lambda}(p)>0$ for $\nu$ a.\,e.\,$\lambda$, we infer
\[
\int e^{\lambda\, t}\, w_{\lambda}(p)\, d\nu\ge e^{t\, \max K}\, \min_{K}w_{\lambda}(p)\, \nu(K),
\]
contradicting \eqref{fin} for $t\to -\infty$, due to $\max K<0$. The rest of the proof follows verbatim as in the previous case, showing that $\nu(]0, +\infty[)=0$ as well. Therefore ${\rm supp}(\nu)=\{0\}$ and thus $u$ is harmonic and stationary.

\end{proof}

\begin{remark}\label{finalR}
In the case $\kappa>0$ we actually proved the following statement. If there is a point $p\in M$ such that $u(p, t)\le e^{o(-t)}$ as $t\to -\infty$ and $t_{0}$ such that $u(x, t_{0})\le e^{o({\rm d}(x, p))}$ for ${\rm d}(x, p)\to +\infty$, then $u$ is a positive harmonic function.
\end{remark}
\vskip10pt
\noindent
{\small {\bf Acknowledgement.} The author is a member of GNAMPA (Gruppo Nazionale per l'Analisi Matematica, la Probabilit\`a e le loro Applicazioni) of INdAM (Istituto Nazionale di Alta Matematica 'Francesco Severi') and is supported by the grant PRIN n.\ 2017AYM8XW: {\em Non-linear Differential Problems via Variational, Topological and Set-valued Methods} and the grant PdR 2018-2020 - linea di intervento 2: {\em Metodi Variazionali ed Equazioni Differenziali} of the University of Catania.}

\end{document}